\documentclass[10pt]{amsart}
\usepackage{amssymb,mathrsfs,skak}%,MnSymbol}
\def\bN {\mathbf{N}}

\def\bR {\mathbf{R}}

\def\fS {\mathfrak{S}}
\def\fZ {\mathfrak{Z}}

\def\cE {\mathcal{E}}

\def\cK {\mathcal{K}}

\def\cM {\mathcal{M}}

\def\cP {\mathcal{P}}

\def\cX {\mathcal{X}}

\def\scrL{\mathscr{L}}

\def\a {{\alpha}}
\def\b {{\beta}}

\def\de {{\delta}}

\def\th {{\theta}}

\def\l {{\lambda}}

\def\si {{\sigma}}

\def\om {{\omega}}

\def\d {{\partial}}
\def\grad {{\nabla}}

\def\rstr {{\big |}}

\def\la {\langle}
\def\ra {\rangle}
\def \La {\bigg\langle}
\def \Ra {\bigg\rangle}

\newcommand{\Div}{\operatorname{div}}

\newcommand{\Span}{\operatorname{span}}

\newcommand{\Dist}{\operatorname{dist}}
\newcommand{\DDist}{\operatorname{Dist}}

\newcommand{\Lip}{\operatorname{Lip}}

\newcommand{\ba}{\begin{aligned}}
\newcommand{\ea}{\end{aligned}}

\newcommand{\be}{\begin{equation}}
\newcommand{\ee}{\end{equation}}

\newcommand{\lb}{\label}

\newtheorem{theorem}{Theorem}[section]
\newtheorem{corollary}{Corollary}

\newtheorem{lemma}[theorem]{Lemma}
\newtheorem{proposition}{Proposition}

\theoremstyle{definition}

\begin{document}

\title{Empirical measures and Vlasov Hierarchies}

\author[F. Golse]{Fran\c cois Golse}
\address[F.G.]{Ecole Polytechnique, Centre de Math\'ematiques L. Schwartz, 91128 Palaiseau Cedex, France}
\email{francois.golse@math.polytechnique.fr}

\author[C. Mouhot]{Cl\'ement Mouhot}
\address[C.M.]{University of Cambridge, DPMMS, Centre for Mathematical Sciences, Wilberforce road, Cambridge CB30WA, United Kingdom}
\email{C.Mouhot@dpmms.cam.ac.uk}

\author[V. Ricci]{Valeria Ricci}
\address[V.R.]{Dipartimento di Matematica e Informatica, Universit\`a di Palermo, Via Archirafi 34, 90123 Palermo Italy}
\email{valeria.ricci@unipa.it}

\begin{abstract}
The present note reviews some aspects of the mean field limit for Vlasov type equations with Lipschitz continuous interaction kernel. We discuss in particular the connection between the approach 
involving the $N$-particle empirical measure and the formulation based on the BBGKY hierarchy. This leads to a more direct proof of the quantitative estimates on the propagation of chaos 
obtained on a more general class of interacting systems in [S.Mischler, C. Mouhot, B. Wennberg, arXiv:1101.4727]. Our main result is a stability estimate on the BBGKY hierarchy uniform in the 
number of particles, which implies a stability estimate in the sense of the Monge-Kantorovich distance with exponent $1$ on the infinite mean field hierarchy. This last result amplifies Spohn's 
uniqueness theorem [H. Spohn, Math. Meth. Appl. Sci. \textit{3} (1981), 445--455].
\end{abstract}



\maketitle

\rightline{\emph{In memory of our friend and colleague Seiji Ukai (1940-2012)}}
\bigskip

%%%%%%%%%%%%%%%%%%%%%%%%%%%%%%%%%%%%%%%%%%%%%%%%%%%%%%%%%%%%%%%%%%%%%%%%%%%%%%%%%%%%%%%%%%%%%%%%

\section*{Introduction}

%%%%%%%%%%%%%%%%%%%%%%%%%%%%%%%%%%%%%%%%%%%%%%%%%%%%%%%%%%%%%%%%%%%%%%%%%%%%%%%%%%%%%%%%%%%%%%%%

Mean field evolution PDEs are an important class of models in non equibrium statistical mechanics. Perhaps the main example in this class of models is the Vlasov-Poisson system. It takes the 
form
\begin{equation}\label{VlasovP}
(\d_t+\xi\cdot\grad_x)f(t,x,\xi)+E[f](t,x)\cdot\grad_\xi f(t,x,\xi)=0\,,\quad x,\xi\in\bR^3\,,
\end{equation}
where $f\equiv f(t,x,\xi)\ge 0$ is the (unknown) distribution function of a particle system (in other words, the number density at time $t$ of particles located at the position $x$ with velocity $\xi$), 
while
$$
E[f](t,x):=-\iint_{\bR^3\times\bR^3}\grad V(x-y)f(t,x,\xi)d\xi dy\,.
$$
The Vlasov-Poisson system is a fundamental model in plasma physics \cite{KrallTriv}; it is obtained by specializing (\ref{VlasovP}) to the case where $V$ is the Coulomb potential
$$
V(x):=\frac1{4\pi|x|}\,,\qquad x\in\bR^3\setminus\{0\}\,.
$$
In that case, $f$ is the distribution function of a system of identical charged point particles, and $E[f]$ is the self-consistent electrostatic field created by those particles.

Another important example is the vorticity formulation of the Euler equation for incompressible fluids in space dimension $d=2$, which takes the form
\begin{equation}\label{Euler2D}
\d_t\om(t,x)+\Div_x(\om u)(t,x)=0\,,\quad x\in\bR^2\,.
\end{equation}
Here $\om\equiv\om(t,x)\in\bR$ is the unknown vorticity field, while
$$
u(t,x):=\int_{\bR^2}K(x-y)\om(t,y)dy
$$
is the velocity field. The integral kernel $K$ is given by the formula
\begin{equation}\label{KEuler}
K(z):=\tfrac1{2\pi}J\grad\ln|z|\,,\quad\hbox{ with }J:=\left(\begin{matrix}0&-1\\1&0\end{matrix}\right)\,.
\end{equation}
While the vorticity field $\om$ is in general of indefinite sign and cannot be viewed as a density of particles, applying the methods of statistical mechanics to a gas of vortices provides valuable 
information on the dynamics of incompressible, inviscid fluids in space dimension $2$ \cite{Onsa}.

A fundamental question in nonequilibrium statistical mechanics is to derive mean field PDEs such as (\ref{VlasovP}) or (\ref{Euler2D}) rigorously from the dynamics of finite particle systems in 
some appropriate limit. This remains an open problem at the time of this writing, at least for interactions such as the Coulomb potential $V$ or the Biot-Savart kernel $K$, which are both singular 
at the origin. The case of interactions with a singularity at the origin weaker than either the Coulomb potential or the Biot-Savart kernel has been recently considered in \cite{HauJab1,HauJab2}. 
In the case of regularized interactions, the corresponding limits have been established rigorously already some time ago \cite{Neunzert,BraunHepp,Dobrushin79}. 

There are two different ways of handling this problem. 

One can prove that the empirical measure of a system of $N$ identical, interacting particles converges to the solution to the mean field PDE as $N\to\infty$: this is the approach used in 
\cite{Neunzert,BraunHepp,Dobrushin79}. Alternately, one can try to use BBGKY hierarchies and establish the propagation of chaos in the limit $N\to\infty$: see \cite{Spohn}.

In this short note, we explain how both approaches are related in the case of smooth interaction kernels (see Theorem \ref{T-Chaotic}), and obtain a stability estimate on the BBGKY hierarchy 
that is uniform in the number of particles (see Theorem \ref{T-StabBBGKY}). As a consequence, we prove the continuous dependence on initial data of statistical solutions of the mean field PDE, 
with an estimate in some appropriate Monge-Kantorovich distance, thereby amplifying Spohn's uniqueness theorem for solutions of the infinite Vlasov hierarchy in \cite{Spohn}. (The notion of
statistical solution of the mean field equation will be recalled at the end of section \ref{S-ContDep}.)

Seiji Ukai contributed several famous results in the theory of PDEs. For instance, his note \cite{Ukai74} is the first global existence and uniqueness theorem on the Cauchy problem for the 
Boltzmann equation. In addition to his impressive work on the Boltzmann equation, Seiji Ukai is also at the origin of the regularity theory for the Vlasov-Poisson system \cite{UkaiOkabe}. More
recently, Seiji Ukai also gave a striking interpretation of the derivation of the Boltzmann equation from the $N$-body problem in classical mechanics in terms of the Nirenberg-Ovsyannikov 
abstract variant of the Cauchy-Kovalevskaya theorem \cite{UkaiCK}. The present paper discusses different aspects of the analogous question for Vlasov type equations, and is dedicated to 
Seiji Ukai's memory, in recognition of his considerable influence on the field of kinetic models.

%%%%%%%%%%%%%%%%%%%%%%%%%%%%%%%%%%%%%%%%%%%%%%%%%%%%%%%%%%%%%%%%%%%%%%%%%%%%%%%%%%%%%%%%%%%%%%%%

\section{Vlasov Equations and Mean-Field Limit}\lb{S-Empir}

%%%%%%%%%%%%%%%%%%%%%%%%%%%%%%%%%%%%%%%%%%%%%%%%%%%%%%%%%%%%%%%%%%%%%%%%%%%%%%%%%%%%%%%%%%%%%%%%

Let $K:\,\bR^d\times\bR^d\to\bR^d$ be a continuous map satisfying
\begin{equation}\label{K0}
K(z,z')+K(z',z)=0\,,\quad z,z'\in\bR^d\,,
\end{equation}
and the Lipschitz condition
\begin{equation}\label{LipK}
\left\{
\begin{aligned}
\sup_{z'\in\bR^d}|K(z_1,z')\!-\!K(z_2,z')|\le L|z_1-z_2|\,,
\\
\sup_{z\in\bR^d}\,|\,K(z,z_1)-K(z,z_2)|\le L|z_1-z_2|\,.
\end{aligned}
\right.
\end{equation}
These two conditions imply in particular that
\begin{equation}\label{BoundK}
|K(z,z')|\le L|z-z'|\le L(|z|+|z'|)\,,\quad z,z'\in\bR^d\,.
\end{equation}

Consider a system of $N$ identical particles, the state of the $k$th particle at time $t$ being defined by $z_k(t)\in\bR^d$. Assume that the evolution of this system of particles is governed by the 
$N$-body ODE system
\begin{equation}\label{NODE}
\left\{
\begin{aligned}
{}&\dot{z}_k(t)=\frac1N\sum_{l=1}^NK(z_k(t),z_l(t)) \,,\quad k=1,\ldots,N\,,\quad t\in\bR\,,
\\
&z_k(0)=z_k^{in}\,.
\end{aligned}
\right.
\end{equation}
We have assumed a mean field scaling: the interaction between two particles in the system is of order $1/N$, so that the collective action of the particle system on each particle is of order unity.

Together with the $N$-body ODE system (\ref{NODE}), we consider the mean field Vlasov equation
\begin{equation}\label{VlasovK}
\left\{
\begin{aligned}
{}&\d_tf(t,z)+\Div_z(f(t,z)\cK f(t,z))=0\,,\quad z\in\bR^d\,,\,\,t\in\bR\,,
\\
&f\rstr_{t=0}=f^{in}\,,
\end{aligned}
\right.
\end{equation}
where $f^{in}\in L^1(\bR^d;(1+|z|)dz)$. In equation (\ref{VlasovK}), we abuse the notation $\cK f(t,z)$ to designate $(\cK f(t,\cdot))(z)$, where 
\begin{equation}\label{DefKphi}
\cK\phi(z):=\int_{\bR^d}K(z,z')\phi(z')dz'\,,\quad\phi\in L^1(\bR^d;(1+|z|)dz) \,.
\end{equation}

The vorticity formulation of the incompressible Euler equation (\ref{Euler2D}) in space dimension $2$ is an obvious example of mean field equation (\ref{VlasovK}), with integral kernel given 
by (\ref{KEuler}).

The Vlasov equation (\ref{VlasovP}) can also be put in this form. In that case, $d=6$ and the $N$-body ODE system (\ref{NODE}) is the system of Hamilton's equations 
$$
\dot{\xi}_k=-\d_{x_k}H_N(x_1,\xi_1,\ldots,x_N,\xi_N)\,,\quad\dot{x}_k=\d_{\xi_k}H_N(x_1,\xi_1,\ldots,x_N,\xi_N)\,,
$$
where the Hamiltonian $H_N$ is defined on $(\bR^6)^N$ by the formula
\begin{equation}\label{NHamilt}
H_N(x_1,\xi_1,\ldots,x_N,\xi_N):=\tfrac12\sum_{k=1}^N|\xi_k|^2+\frac1N\sum_{k,l=1}^NV(x_k-x_l)\,.
\end{equation}
Denoting $z_k=(x_k,\xi_k)\in\bR^3\times\bR^3\simeq\bR^6$ the $k$-th pair of conjugate variables, the interaction kernel $K$ in this case is given by 
\begin{equation}\label{KVlasov}
K((x,\xi),(x',\xi'))=(\xi-\xi',-\grad V(x-x'))
\end{equation}
and satisfies the assumptions (\ref{K0}) and (\ref{LipK}) if and only if 
\begin{equation}\label{HypV}
V\in C^1(\bR^3)\,,\quad\grad V\in\Lip(\bR^3;\bR^3)\hbox{ and }\grad V(z)+\grad V(-z)=0\hbox{ for all }z\in\bR^3\,.
\end{equation}
Since the total mass and momentum
$$
\iint_{\bR^3\times\bR^3}f(t,x,\xi)dxd\xi\quad\hbox{ and }\iint_{\bR^3\times\bR^3}\xi f(t,x,\xi )dxd\xi 
$$
are invariant under the dynamics defined by (\ref{VlasovP}), the mean field PDE (\ref{VlasovK}) coincides with (\ref{VlasovP}) for distribution functions $f$ such that
$$
\iint_{\bR^3\times\bR^3}f(0,x,\xi)dxd\xi=1\quad\hbox{ and }\iint_{\bR^3\times\bR^3}\xi f(0,x,\xi)dxd\xi=0\,.
$$
Notice that neither the vorticity formulation of the incompressible Euler equation in space dimension $2$ nor the Vlasov-Poisson system satisfy the Lipschitz condition (\ref{LipK}) on their 
interaction kernels $K$, because of the singularity of the Biot-Savart kernel or of the Coulomb potential at the origin. However, both these equations do satisfy the antisymmetry condition 
(\ref{K0}) on their interaction kernels.

Condition (\ref{BoundK}) implies that the differential system (\ref{NODE}) has a global solution defined for all $t\in\bR$, denoted by
$$
(z_1(t),\ldots,z_N(t))=:T^N_t(z_1^{in}\ldots,z_N^{in})
$$
for all initial data $z_1^{in},\ldots,z_N^{in}\in\bR^d$. Henceforth, we denote
$$
Z_N:=(z_1,\ldots,z_N)\,,\qquad Z_N^{in}:=(z_1^{in}\ldots,z_N^{in})\,.
$$

To each $Z_N\in(\bR^d)^N$ we associate the empirical measure
\begin{equation}\label{EmpirM}
\mu_{Z_N}:=\frac1N\sum_{k=1}^N\de_{z_k}\,,
\end{equation}
and, for each $t\in\bR$, we define
\begin{equation}\label{EmpirMt}
\mu_{N}(t):=\mu_{T^N_tZ_N^{in}}\,,\quad t\in\bR\,.
\end{equation}
For each $Z_N^{in}\in(\bR^d)^N$, the time-dependent empirical measure $\mu_N$ defined in (\ref{EmpirMt}) belongs to $C(\bR_+,w-\cP_1(\bR^d))$, where $\cP_r(\bR^d)$ designates, 
for each $r>0$, the set of Borel probability measures $\mu$ on $\bR^d$ such that
\begin{equation}\label{DefPs}
\int_{\bR^d}|z|^r\mu(dz)<\infty\,.
\end{equation}

A remarkable property of the $N$-body ODE (\ref{NODE}) is that the time-dependent empirical measure $\mu_N$ is a weak solution to the Vlasov equation (\ref{VlasovK}), where the 
operator $\cK$ in (\ref{DefKphi}) is extended to $\cP_1(\bR^d)$ by the formula
\begin{equation}\label{DefKmu}
\cK\mu(z)=\int_{\bR^d}K(z,z')\mu(dz')\,.
\end{equation}

We first recall the following important result (proved in the case (\ref{KVlasov})).

\begin{proposition}[Dobrushin \cite{Dobrushin79}]\label{P-Dobru}
For all $\mu^{in}\in\cP_1(\bR^d)$, the Cauchy problem
\begin{equation}\label{VlasovMut}
\left\{
\begin{aligned}
{}&\d_t\mu(t)+\Div(\mu(t)\cK\mu(t))=0\,,\quad t\in\bR\,,
\\
&\mu(0)=\mu^{in}\,,
\end{aligned}
\right.
\end{equation}
has a unique weak solution $\mu\in C(\bR;w-\cP_1(\bR^d))$. Moreover

\smallskip
\noindent
(a) For all $Z_N^{in}\in(\bR^d)^N$, the unique solution to (\ref{VlasovMut}) in $C(\bR;w-\cP_1(\bR^d))$ with initial data $\mu_{Z_N^{in}}$ is the time dependent empirical measure $\mu_N(t)$ 
defined by (\ref{EmpirMt}).

\noindent
(b) Let $\mu$ and $\nu$ be the solutions of (\ref{VlasovMut}) in $C(\bR;w-\cP_1(\bR^d))$ with initial data respectively $\mu^{in}$ and $\nu^{in}$. Then the Monge-Kantorovich distance 
$\Dist_{MK,1}$ between $\mu(t)$ and $\nu(t)$ satisfies the inequality
\begin{equation}\label{Dobr<}
\Dist_{MK,1}(\mu(t),\nu(t))\le e^{2L|t|}\Dist_{MK,1}(\mu^{in},\nu^{in})\,.
\end{equation}

\noindent
(c) If $\mu^{in}$ is absolutely continuous with respect to the Lebesgue measure $\scrL^d$ of $\bR^d$, then the solution $\mu(t)$ of (\ref{VlasovMut}) is also absolutely continuous with respect 
to $\scrL^d$ for all $t\in\bR$. Thus $\mu(t)$ is of the form $\mu(t)=f(t,\cdot)\scrL^d$ for all $t\in\bR$, where $f$ is the solution to (\ref{VlasovK}).
\end{proposition}

\smallskip
We recall that, for all $r>0$, the Monge-Kantorovich distance $\Dist_{MK,r}$  is defined on $\cP_r(\bR^d)$ by the formula
$$
\Dist_{MK,r}(\mu,\nu)=\inf_{\pi\in\Pi(\mu,\nu)}\left(\iint_{\bR^d\times\bR^d}|x-y|^r\pi(dxdy)\right)^{1/r}\,,
$$
where $\Pi(\mu,\nu)$ is the set of Borel probability measures on $\bR^d\times\bR^d$ such that
$$
\iint_{\bR^d\times\bR^d}(\phi(x)+\psi(y))\pi(dxdy)=\int_{\bR^d}\phi(x)\mu(dx)+\int_{\bR^d}\psi(y)\nu(dy)
$$
for all $\phi,\psi\in C_b(\bR^d)$. In the particular case $r=1$, 
\begin{equation}\label{DualKR}
\Dist_{MK,1}(\mu,\nu)=\sup_{\Lip(\phi)\le 1}\left|\int_{\bR^d}\phi(z)\mu(dz)-\int_{\bR^d}\phi(z)\nu(dz)\right|\,,
\end{equation}
where 
$$
\Lip(\phi):=\sup_{x\not=y\in\bR^d}\frac{|\phi(x)-\phi(y)|}{|x-y|}\,.
$$
(See Theorems 1.14 and 7.3 (i) in \cite{VillaniTOT}.)

The interested reader is referred to the original article \cite{Dobrushin79} for a proof of Proposition \ref{P-Dobru}. We just sketch below the argument for obtaining the upper bound on moments 
of weak solutions of the Vlasov equation (\ref{VlasovMut}). For each $r\ge 1$, one has
\begin{equation*}
\begin{aligned}
\frac{d}{dt}\int_{\bR^d}|z|^r\mu(t,dz)&=r\int_{\bR^d}|z|^{r-2}z\cdot\cK\mu(t,z)\mu(t,dz)
\\
&\le r\iint_{\bR^d\times\bR^d}|z|^{r-1}|K(z,z')|\mu(t,dz)\mu(t,dz')
\\
&\le Lr\iint_{\bR^d\times\bR^d}|z|^{r-1}(|z|+|z'|)\mu(t,dz)\mu(t,dz')
\\
&\le L\iint_{\bR^d\times\bR^d}((2r-1)|z|^r+|z'|^r)\mu(t,dz)\mu(t,dz')
\\
&=2Lr\int_{\bR^d}|z|^r\mu(t,dz)\,. 
\end{aligned}
\end{equation*}
(Here we have used the elementary inequality
\begin{equation}\label{CvxExp}
a^{(r-1)}b\le(1-\tfrac1r)a^r+\tfrac1rb^r\,,\quad a,b>0\,,\,\,r\ge 1\,,
\end{equation}
which is a consequence of the convexity of $z\mapsto e^z$.) By Gronwall's lemma, 
\begin{equation}\label{1stMom}
\int_{\bR^d}|z|^r\mu(t,dz)\le e^{2Lr|t|}\int_{\bR^d}|z|^r\mu^{in}(dz)\,,\quad t\in\bR\,.
\end{equation}

\smallskip
A consequence of Dobrushin's result is that the Vlasov equation (\ref{VlasovK}) governs the mean field limit of the $N$-particle ODE system (\ref{NODE}), a result already established in 
\cite{BraunHepp}\footnote{The argument in \cite{BraunHepp} involves a distance very similar to $\Dist_{MK,1}$, defined by the Kantorovich-Rubinstein formula (\ref{DualKR}), where the 
supremum is taken on the set of all bounded Lipschitz continuous functions $\phi$ such that $\|\phi\|_{L^\infty}+\Lip(\phi)\le 1$: see formulas (2.8-9) in \cite{BraunHepp}. However the proof 
of the mean field limit (Theorem 3.1 in \cite{BraunHepp}) includes a weak convergence argument on the initial data that is not quantitative and differs from Dobrushin's.} --- see also the 
earlier reference \cite{Neunzert}. 

\begin{corollary}
Let $f^{in}$ be a probability density on $\bR^d$ satisfying $f^{in}\in L^1(\bR^d,|z|dz)$. Let $\fZ$ denote a map associating to each $N\in\bN^*$ an $N$-tuple $\fZ(N)$ of $\bR^d$ such that
$$
\mu_{\fZ(N)}\to f^{in}\scrL^d\hbox{ weakly, and }\int_{\bR^d}|z|\mu_{\fZ(N)}(dz)\to\int_{\bR^d}|z|f^{in}(z)dz
$$
as $N\to\infty$. Then the sequence of solutions $T^N_t\fZ(N)$ of the $N$-particle ODE system (\ref{NODE}) satisfies
$$
\mu_{T^N_t\fZ(N)}\to f(t,\cdot)\scrL^d\hbox{ weakly as }N\to\infty\,.
$$
\end{corollary}

\begin{proof}
Since the distance $\Dist_{MK,1}$ metricizes the topology of weak convergence on $\cP_1(\bR^d)$ (see Theorem 7.12 in \cite{VillaniTOT}), the two conditions on $\fZ(N)$ imply that 
$$
\Dist_{MK,1}(\mu_{\fZ(N)},f^{in}\scrL^d)\to 0\quad\hbox{ as }N\to\infty\,.
$$
By statement (b) in the Proposition above, 
$$
\Dist_{MK,1}(\mu_{T^N_t\fZ(N)},f(t,\cdot)\scrL^d)\to 0\quad\hbox{ as }N\to\infty\,,
$$
where $f$ is the solution to (\ref{VlasovK}), which implies the conclusion.
\end{proof}

%%%%%%%%%%%%%%%%%%%%%%%%%%%%%%%%%%%%%%%%%%%%%%%%%%%%%%%%%%%%%%%%%%%%%%%%%%%%%%%%%%%%%%%%%%%%%%%%

\section{Vlasov Hierarchy and the Mean-Field Limit}\lb{S-Hier}

%%%%%%%%%%%%%%%%%%%%%%%%%%%%%%%%%%%%%%%%%%%%%%%%%%%%%%%%%%%%%%%%%%%%%%%%%%%%%%%%%%%%%%%%%%%%%%%%

Another approach to the problem of deriving the Vlasov equation (\ref{VlasovK}) as the mean field limit of the $N$-particle ODE system (\ref{NODE}) involves the Vlasov hierarchy.

Let $P_N^{in}\in\cP_1((\bR^d)^N)$ be a probability measure on the $N$-particle phase-space. All particles being identical, assume that $P_N^{in}$ is symmetric, meaning that, for each 
permutation $\si\in\fS_N$, one has\footnote{If $T:\,X\to Y$ is a measurable map and $m$ is a measure on $X$, $T\#m$ designates the push-forward of $m$ under $T$, that is a measure 
on $Y$, defined as follows: for each positive measurable function $f$ on $Y$,
$$
\int_Yf(y)T\#m(dy)=\int_Xf(T(x))m(dx)\,.
$$}
\begin{equation}\label{SymPin}
S_\si\#P_N^{in}=P_N^{in}\,,
\end{equation}
where 
\begin{equation}\label{DefSsi}
S_\si:\,(\bR^d)^N\ni(x_1,\ldots,x_N)\mapsto(x_{\si(1)},\ldots,x_{\si(N)})\in(\bR^d)^N\,. 
\end{equation}
Consider the time dependent $N$-particle probability distribution $P_N(t)$ defined by the formula 
\begin{equation}\label{DefPNt}
P_N(t):=T^N_t\#P_N^{in}\,.
\end{equation}
It is the unique solution in $C(\bR;w-\cP_1((\bR^d)^N))$ of the $N$-body Liouville equation
\begin{equation}\label{NLiouv}
\left\{
\begin{aligned}
{}&\d_tP_N(t)+\frac1N\sum_{k,l=1}^N\Div_{z_k}(P_N(t)K(z_k,z_l))=0\,,
\\
&P_N\rstr_{t=0}=P_N^{in}\,.
\end{aligned}
\right.
\end{equation}
Observe that the first order differential operator on $(\bR^d)^N$
$$
\frac1N\sum_{k,l=1}^NK(z_k,z_l)\d_{z_l}\quad\hbox{ commutes with }S_\si\,.
$$
Therefore, the symmetry (\ref{SymPin}) is propagated by the $N$-body dynamics, so that
\begin{equation}\label{SymPNt}
S_\si\#P_N(t)=P_N(t)\,,\quad\hbox{ for all }\si\in\fS_N\hbox{ and all }t\in\bR\,.
\end{equation}

For each $m=1,\ldots,N-1$, define the $m$-body marginal of $P_N(t)$ as
\begin{equation}\label{DefPNm}
\left\{
\begin{aligned}
{}&P_{N:m}(t):=p_m\#P_N(t)\,,&&\quad m=1,\ldots,N\,,
\\
&P_{N:m}(t):=0\,,&&\quad m\ge N\,,
\end{aligned}
\right.
\end{equation}
where 
\begin{equation}\label{Defpm}
p_m:\,(\bR^d)^N\ni(z_1,\ldots,z_N)\mapsto(z_1,\ldots,z_m)\in(\bR^d)^m\,.
\end{equation}
The sequence of $m$-body marginals of $P_N$ satisfies the BBGKY hierarchy of equations
\begin{equation}\label{BBGKY}
\begin{aligned}
\d_tP_{N:m}(t)&+\frac{N-m}N\sum_{k=1}^m\Div_{z_k}\left(\int_{\bR^d}K(z_k,z_{m+1})P_{N:m+1}(t,dz_{m+1})\right)
\\
&+\frac1N\sum_{k,l=1}^m\Div_{z_k}(P_{N:m}(t)K(z_k,z_l))=0\,,\qquad m\ge 1\,.
\end{aligned}
\end{equation}
Since the equation for $m=N$ in this hierarchy coincides with the $N$-body Liouville equation (\ref{NLiouv}), the BBGKY hierarchy (\ref{BBGKY}) is exactly equivalent to (\ref{NLiouv}).

Assume that $P_N^{in}=(f^{in}\scrL^d)^{\otimes N}$, where $f^{in}$ is a probability density on $\bR^d$ belonging to $L^1(\bR^d;|z|dz)$. The sequence\footnote{More precisely, this is 
the sequence indexed by $N\ge 1$ of the elements $(P_{N:m})_{m\ge 1}$ of the product space $\cX$ (as explained in formula (\ref{DefPNm}), for $m>N\ge 1$, one has $P_{N:m}=0$).} 
$((P_{N:m})_{m\ge 1})_{N\ge 1}$ is relatively compact in the product space
$$
\cX=\prod_{m\ge 1}L^\infty(\bR;\cM((\bR^d)^m))
$$
for the product topology, each factor being endowed with the weak-* topology of the dual of $L^1(\bR;C_b((\bR^d)^m))$. Each limit point of that family as $N\to\infty$ is a solution 
$(P_m)_{m\ge 1}$ of the (infinite) Vlasov hierarchy
\begin{equation}\label{VlasovInfin}
\left\{
\begin{aligned}
{}&\d_tP_{m}(t)+\sum_{k=1}^m\Div_{z_k}\left(\int_{\bR^d}K(z_k,z_{m+1})P_{m+1}(t,dz_{m+1})\right)=0\,,\qquad m\ge 1\,,
\\
&P_m\rstr_{t=0}=(f^{in}\scrL^d)^{\otimes m}\,.
\end{aligned}
\right.
\end{equation}
(See \cite{BGGM2000} for a proof of this result in the genuine Vlasov case (\ref{KVlasov}).) We just sketch below the formal argument analogous to (\ref{1stMom}) implying tightness of the 
sequence $(P_{N:m}(t))_{N\ge m}$ for $t$ and $m$ fixed, as $N\to\infty$. For each $N,r\ge 1$, one has 
\begin{equation*}
\begin{aligned}
\frac{d}{dt}\int_{(\bR^d)^N}|z_1|^rP_N(t,dZ_N)=\frac{r}N\sum_{j=1}^N\int_{(\bR^d)^N}|z_1|^{r-2}z_1\cdot K(z_1,z_j)P_N(t,dZ_N)&
\\
\le\frac{Lr}N\sum_{j=1}^N\int_{(\bR^d)^N}|z_1|^{r-1}(|z_1|+|z_j|)P_N(t,dZ_N)&
\\
=Lr\int_{(\bR^d)^N}|z_1|^rP_N(t,dZ_N)+\frac{Lr}N\sum_{j=1}^N\int_{(\bR^d)^N}|z_1|^{r-1}|z_j|P_N(t,dZ_N)&
\\
\le Lr\int_{(\bR^d)^N}|z_1|^rP_N(t,dZ_N)+\frac{L}N\sum_{j=1}^N\int_{(\bR^d)^N}((r-1)|z_1|^r+|z_j|^r)P_N(t,dZ_N)&
\\
= (2r-1)L\int_{(\bR^d)^N}|z_1|^rP_N(t,dZ_N)+\frac{L}N\sum_{j=1}^N\int_{(\bR^d)^N}|z_j|^rP_N(t,dZ_N)&
\\
=2Lr\int_{(\bR^d)^N}|z_1|^rP_N(t,dZ_N)&\,.
\end{aligned}
\end{equation*}
(The last inequality above uses again (\ref{CvxExp}), while the last equality uses the symmetry property (\ref{SymPNt}).) By Gronwall's inequality, for each $t\in\bR$ and $N\ge 1$, one has
\begin{equation}\label{rMomPNm}
\begin{aligned}
\int_{(\bR^d)^N}|z_1|^rP_N(t,dZ_N)&\le e^{2Lr|t|}\int_{(\bR^d)^N}|z_1|^rP_N(0,dZ_N)
\\
&=e^{2Lr|t|}\int_{\bR^d}|z|^rf^{in}(z)dz\,.
\end{aligned}
\end{equation}
In particular, for each $t\in\bR$, $m\in\bN^*$ and $N\ge m$, one has
\begin{equation*}
\begin{aligned}
\int_{(\bR^d)^m}\sum_{j=1}^m|z_j|^rP_{N:m}(t,dZ_m)&=\int_{(\bR^d)^N}\sum_{j=1}^m|z_j|^rP_N(t,dZ_N)
\\
&=m\int_{(\bR^d)^N}|z_1|^rP_N(t,dZ_N)
\\
&\le me^{2Lr|t|}\int_{\bR^d}|z|^rf^{in}(z)dz\,.
\end{aligned}
\end{equation*}

If $f$ is a solution to the Vlasov equation (\ref{VlasovK}), the sequence $((f\scrL^d)^{\otimes m})_{m\ge 1}$ is a solution to the infinite Vlasov hierarchy (\ref{VlasovInfin}). Thus, if one knows 
that each limit point $(P_m)_{m\ge 1}$ of the family $((P_{N:m})_{m\ge 1})_{N\ge 1}$ belongs to a functional space where the infinite Vlasov hierarchy (\ref{VlasovInfin}) has only one solution, 
one concludes that 
$$
P_{N:m}(t)\to (f(t)\scrL^d)^{\otimes m}\hbox{ weakly in }\cP((\bR^d)^m)\hbox{ as }N\to\infty
$$
for each $m\ge 1$, where $f$ is the solution to the Vlasov equation (\ref{VlasovK}). 

Uniqueness of the solution to the infinite Vlasov hierarchy has been established in the genuine Vlasov case (\ref{KVlasov}), first by Narnhofer-Sewell \cite{NarnhoSewell} in the case where 
the potential $V$ is analytic with Fourier transform $\hat V\in C_c(\bR^d)$, and later by Spohn \cite{Spohn} in the more general case where $\hat V$ is a Radon measure on $\bR^d$ such that 
$$
\int_{\bR^d}|k|^2|\hat V(k)|dk<\infty\,.
$$
Notice however that Spohn's uniqueness theorem uses Dobrushin's Proposition \ref{P-Dobru} (or equivalently the mean field limit obtained in both \cite{BraunHepp,Neunzert}). Therefore the 
approach based on the BBGKY hierarchy is not really an alternative to the one based on the empirical measure.

These two approaches of the mean field limit involve objects of a different nature. Indeed, the time dependent empirical measures considered in the first approach are measures defined on the 
single-particle phase space, whereas the second approach based on the BBGKY hierarchy involves the sequence of $m$-particle phase spaces for all $m\ge 1$.

%%%%%%%%%%%%%%%%%%%%%%%%%%%%%%%%%%%%%%%%%%%%%%%%%%%%%%%%%%%%%%%%%%%%%%%%%%%%%%%%%%%%%%%%%%%%%%%

\section{Empirical Measures and Chaotic Sequences}

%%%%%%%%%%%%%%%%%%%%%%%%%%%%%%%%%%%%%%%%%%%%%%%%%%%%%%%%%%%%%%%%%%%%%%%%%%%%%%%%%%%%%%%%%%%%%%%

The two approaches of the mean field limit sketched in sections \ref{S-Empir} and \ref{S-Hier} involve probability measures defined on very different phase spaces, and therefore may seem a 
priori unrelated. Indeed, in section \ref{S-Empir}, the Vlasov equation (\ref{VlasovK}) is the equation governing the weak limit of the sequence of empirical measures $\mu_{T^N_tZ_N^{in}}$ 
viewed as probability measures on the $1$-particle phase space $\bR^d$. In section \ref{S-Hier}, the object of interest is $P_N$, which is a symmetric probability measures on the $N$-particle 
phase space $(\bR^d)^N$, and the goal of the mean field limit is to describe the asymptotic behavior of $P_N$ as $N\to\infty$ through the sequence of its marginals $P_{N:m}$. The $m$-th 
marginal $P_{N:m}$ of $P_N$ is itself a symmetric probability measure on the $m$-particle phase space $(\bR^d)^m$.

Perhaps the key to understanding how these different approaches to the mean field limit are related is the following observation\footnote{The idea of identifying the $N$-tuple $(z_1,\ldots,z_N)$
modulo permutations of the $z_k$s to the empirical measure $\mu_{(z_1,\ldots,z_N)}$  explicitly appears in \cite{Grunbaum} (see formula (1) on p. 330).}: the empirical measure 
$$
\mu_{(z_1,\ldots,z_N)}(dz):=\frac1N\sum_{k=1}^N\de(z-z_k)
$$
is a symmetric function of the $N$ variables $(z_1,\ldots,z_N)\in(\bR^d)^N$ with values in the set of probability measures in the variable $z\in\bR^d$. With this observation in mind, it becomes
natural to consider expressions of the form
$$
\int_{(\bR^d)^N}\mu^{\otimes m}_{z_1,\ldots,z_N}P_N(dz_1\ldots dz_N)
$$
where the measure-valued symmetric function $\mu^{\otimes m}_{z_1,\ldots,z_N}$ of the $N$-tuple $(z_1,\ldots,z_N)$ is averaged under the symmetric probability measure $P_N$ defined
on the $N$-particle phase space $(\bR^d)^N$. This expression defines a probability measure on the $m$-particle phase space $(\bR^d)^m$, which is related to the $m$-th marginal of $P_N$
by a combinatorial argument that is the key to statement (a) in the theorem below. 

More precisely, statement (a) uses this combinatorial argument together with the $N$-particle dynamics and relates the evolution of the $m$-particle marginal $P_{N:m}$ in the BBGKY hierarchy 
to tensor powers of the empirical measure, which is a measure valued solution of the mean field equation (\ref{VlasovK}). In other words, statement (a) in the theorem below really bridges the
two approaches to the mean field limit.

Combining statement (a) with Dobrushin's inequality and a quantitative variant of the law of large numbers for the initial $1$-particle distribution $f^{in}\scrL^d$, we arrive at a chaoticity estimate,
measuring the distance from $P_{N:m}(t)$ to $(f(t)\scrL^d)^{\otimes m}$ in terms of $N\ge m\ge 1$ for all $t\ge 0$: see statement (c) in the theorem below.

\begin{theorem}\label{T-Chaotic}
Let $P_N^{in}\in\cP_1((\bR^d)^N)$ satisfy the symmetry condition (\ref{SymPin}), and let $t\mapsto P_N(t)$ be defined by (\ref{DefPNt}). Then

\smallskip
\noindent
(a) \textbf{From tensorized empirical measures to marginals} For all $t\in\bR$ and all $N\in\bN^*$
\begin{equation}\label{Empir-BBGKY}
\int_{(\bR^d)^N}\mu_{T^N_tZ_N^{in}}^{\otimes m}P_N^{in}(dZ_N^{in})=\frac{N!}{(N-m)!N^m}P_{N:m}(t)+R_{N,m}(t)\,,
\end{equation}
where $R_{N,m}(t)$ is a positive Radon measure on $(\bR^d)^m$ with total mass
\begin{equation}\label{IneqRNm}
\la R_{N,m}(t),1\ra=1-\frac{N!}{(N-m)!N^m}\le\frac{m(m-1)}{2N}\,.
\end{equation}
(b) \textbf{Dobrushin's estimate for marginals} For all $t\in\bR$, all $N\in\bN^*$ and all $m=1,\ldots,N$, and for all bounded and Lipschitz continuous $\phi_m$ defined on $(\bR^d)^m$, one 
has
\begin{equation*}
\begin{aligned}
{}&|\la P_{N:m}(t)-(f(t)\scrL^d)^{\otimes m},\phi_m\ra|\le\|\phi_m\|_{L^\infty}\frac{m(m-1)}{N}
\\
&\qquad\qquad+me^{2L|t|}\Lip(\phi_m)\int_{(\bR^d)^N}\Dist_{MK,1}(\mu_{Z_N^{in}},f^{in}\scrL^d)P_N^{in}(dZ_N^{in})\,.
\end{aligned}
\end{equation*}
(c) \textbf{Chaoticity estimate} Assume that $P_N^{in}=(f^{in}\scrL^d)^{\otimes N}$ and that
\begin{equation}\label{d+5Momf}
a:=\int_{\bR^d}|z|^{d+5}f^{in}(z)dz<\infty\,;
\end{equation}
then
\begin{equation*}
\begin{aligned}
\|P_{N:m}(t)-(f(t)\scrL^d)^{\otimes m}\|_{W^{-1,1}((\bR^d)^m)}\le m\left(\frac{m-1}{N}+e^{2L|t|}\frac{C(a,d)}{N^{1/(d+4)}}\right)
\end{aligned}
\end{equation*}
for all $t\in\bR$ and all $N\ge m\ge 1$. In particular, for $m=1$, one has
$$
\Dist_{MK,1}(P_{N:1}(t),f(t)\scrL^d)\le C(a,d)e^{2L|t|}/N^{1/(d+4)}\,.
$$
\end{theorem}

\begin{proof} [Proof of statement (a)]
Let $\phi_m\in C_b((\bR^d)^m)$; define
$$
\Phi_m(z_1,\ldots,z_N)=\frac1{N^m}\sum_{j\in F(m,N)}\phi_m(z_{j(1)},\ldots,z_{j(m)})
$$
where $F(m,N)$ is the set of maps from $\{1,\ldots,m\}$ to $\{1,\ldots,N\}$. Thus
$$
\la\mu_{T^N_tZ_N^{in}}^{\otimes m},\phi_m\ra=\Phi_m(T^N_tZ_N^{in})
$$
so that
\begin{equation*}
\begin{aligned}
\La\int_{(\bR^d)^N}\mu_{T^N_tZ_N^{in}}^{\otimes m}P_N^{in}(dZ_N^{in}),\phi_m\Ra
&=
\int_{(\bR^d)^N}\Phi_m(T^N_tZ_N^{in})P_N^{in}(dZ_N^{in})
\\
&=
\int_{(\bR^d)^N}\Phi_m(Z_N)P_N(t,dZ_N)\,.
\end{aligned}
\end{equation*}

Next split the summation defining $\Phi_m$ as
\begin{equation*}
\begin{aligned}
\Phi_m(z_1,\ldots,z_N)&=\frac1{N^m}\sum_{j\in J(m,N)}\phi_m(z_{j(1)},\ldots,z_{j(m)})
\\
&+\frac1{N^m}\sum_{j\in G(m,N)}\phi_m(z_{j(1)},\ldots,z_{j(m)})
=:\Phi_m^J(Z_N)+\Phi_m^G(Z_N)\,,
\end{aligned}
\end{equation*}
where $J(m,N)$ is the set of one-to-one maps from $\{1,\ldots,m\}$ to $\{1,\ldots,N\}$ and $G(m,N):=F(m,N)\setminus J(m,N)$. By the symmetry property (\ref{SymPNt}) of $P_n(t)$, for each 
$j\in J(m,N)$, one has
\begin{equation*}
\begin{aligned}
\int_{(\bR^d)^N}\phi_m(z_{j(1)},\ldots,z_{j(m)})P_N(t,dZ_N)&=\int_{(\bR^d)^N}\phi_m(z_1,\ldots,z_m)P_N(t,dZ_N)
\\
&=\int_{(\bR^d)^m}\phi_m(Z_m)P_{N:m}(t,Z_m)\,,
\end{aligned}
\end{equation*}
so that
$$
\int_{(\bR^d)^N}\Phi_m^J(Z_N)P_N(t,dZ_N)=\frac{\#J(m,N)}{N^m}\int_{(\bR^d)^m}\phi_m(Z_m)P_{N:m}(t,dZ_m)\,.
$$
On the other hand, the formula
\begin{equation}\label{DefRNm}
\la R_{N,m}(t),\phi_m\ra:=\int_{(\bR^d)^N}\Phi_m^G(Z_N)P_N(t,dZ_N)
\end{equation}
defines a positive Radon measure satisfying
$$
\la R_{N,m}(t),1\ra=\frac{\#G(m,N)}{N^m}\,.
$$
The result follows from the equality
$$
\#J(m,N)=N(N-1)\ldots(N-m+1)\,,
$$
and the inequality
\begin{equation}\label{IneqB}
\begin{aligned}
\frac{\#G(m,N)}{N^m}&=1-\left(1-\frac1N\right)\ldots\left(1-\frac{m-1}N\right)
\\
&\le\frac{1+\ldots+m-1}{N}=\frac{m(m-1)}{2N}
\end{aligned}
\end{equation}
following from Theorem 58 in \cite{HardyLittPol}.
\end{proof}

\begin{lemma}
Let $\l\in\cP_1(\bR^m)$ and $\mu,\nu\in\cP_1(\bR^n)$. Then
\begin{equation*}
\left\{
\begin{aligned}
\Dist_{MK,1}(\l\otimes\mu,\l\otimes\nu)&\le\Dist_{MK,1}(\mu,\nu)\,,
\\
\Dist_{MK,1}(\mu\otimes\l,\nu\otimes\l)&\le\Dist_{MK,1}(\mu,\nu)\,.
\end{aligned}
\right.
\end{equation*}
\end{lemma}

\begin{proof}
To $\pi\in\Pi(\mu,\nu)$ we associate $\pi_\l\in\cP((\bR^{m+n})^2)$ defined by
\begin{equation*}
\begin{aligned}
\iint_{\bR^{m+n}}\iint_{\bR^{m+n}}\chi(x_1,x_2,y_1,y_2)\pi_\l(dx_1dx_2dy_1dy_2)
\\
=\int_{\bR^m}\iint_{(\bR^n)^2}\chi(z,x_2,z,y_2)\l(dz)\pi(dx_2dy_2)
\end{aligned}
\end{equation*}
for all $\chi\in C((\bR^{m+n})^2)$ such that 
$$
|\chi(x_1,x_2,y_1,y_2)|=O(|x_1|+|x_2|+|y_1|+|y_2|)\hbox{ as }|x_1|+|x_2|+|y_1|+|y_2|\to\infty\,.
$$
Obviously $\pi_\l\in\Pi(\l\otimes\mu,\l\otimes\nu)$ so that
\begin{equation*}
\begin{aligned}
\Dist_{MK,1}(\l\otimes\mu,\l\otimes\nu)\!\le\!\iint_{\bR^{m+n}}\Dist((x_1,x_2),(y_1,y_2))\pi_\l(dx_1dx_2dy_1dy_2)&
\\
=\iint_{\bR^m}\iint_{(\bR^n)^2}|x_2-y_2|\l(dz)\pi(dx_2dy_2)=\iint_{(\bR^n)^2}|x_2-y_2|\pi(dx_2dy_2)&\,.
\end{aligned}
\end{equation*}
The first inequality in the lemma follows from minimizing the right hand side above for $\pi$ running through $\Pi(\mu,\nu)$. The second inequality in the lemma is established by a similar 
argument.
\end{proof}

\begin{corollary}\label{C-DistProdTens}
For all $\mu,\nu\in\cP_1(\bR^d)$ and all $m\ge 1$, one has
$$
\Dist_{MK,1}(\mu^{\otimes m},\nu^{\otimes m})\le m\Dist_{MK,1}(\mu,\nu)\,.
$$
\end{corollary}

\begin{proof}
By the triangle inequality
\begin{equation}\label{dmumnum}
\begin{aligned}
\Dist_{MK,1}(\mu^{\otimes m},\nu^{\otimes m})&\le\Dist_{MK,1}(\mu^{\otimes m},\mu^{\otimes{(m-1)}}\otimes\nu)
\\
&+\sum_{j=1}^{m-2}\Dist_{MK,1}(\mu^{\otimes(m-1-j)}\otimes\mu\otimes\nu^{\otimes j},\mu^{\otimes{(m-1-j)}}\otimes\nu\otimes\nu^{\otimes j})
\\
&+\Dist_{MK,1}(\mu\otimes\nu^{\otimes(m-1)},\nu\otimes\nu^{\otimes(m-1)})\,.
\end{aligned}
\end{equation}
The general term in the summation above satisfies
\begin{equation*}
\begin{aligned}
\Dist_{MK,1}(\mu^{\otimes(m-1-j)}\otimes\mu\otimes\nu^{\otimes j},\mu^{\otimes{(m-1-j)}}\otimes\nu\otimes\nu^{\otimes j})
\\
\le
\Dist_{MK,1}(\mu\otimes\nu^{\otimes j},\nu\otimes\nu^{\otimes j})\le\Dist_{MK,1}(\mu,\nu)\,,
\end{aligned}
\end{equation*}
where the first inequality follows from the first inequality in the lemma and the second from the second inequality in the lemma. The first and last terms on the right hand side of  (\ref{dmumnum}) 
are estimated by applying respectively the first and the second inequalities in the lemma.
\end{proof}

\begin{proof}[Proof of statements (b) and (c)]
Let $\phi_m\in C_b((\bR^d)^m)$ be Lipschitz continuous; then
\begin{equation*}
\begin{aligned}
\left|\La\frac{N(N-1)\ldots(N-m+1)}{N^m}P_{N:m}(t)-(f(t)\scrL^d)^{\otimes m},\phi_m\Ra\right|&
\\
\le
\left|\int_{(\bR^d)^N}\la\mu_{T^N_tZ_N^{in}}^{\otimes m};\phi_m\ra P_N^{in}(dZ_N^{in})-\la(f(t)\scrL^d)^{\otimes m},\phi_m\ra\right|+\la R_{N,m}(t),|\phi_m|\ra&
\\
=
\left|\int_{(\bR^d)^N}\la\mu_{T^N_tZ_N^{in}}^{\otimes m}-(f(t)\scrL^d)^{\otimes m};\phi_m\ra P_N^{in}(dZ_N^{in})\right|+\la R_{N,m}(t),|\phi_m|\ra&
\\
\le
\int_{(\bR^d)^N}|\la\mu_{T^N_tZ_N^{in}}^{\otimes m}-(f(t)\scrL^d)^{\otimes m};\phi_m\ra|P_N^{in}(dZ_N^{in})+\|\phi_m\|_{L^\infty}\frac{m(m-1)}{2N}&
\\
\le
\Lip(\phi_m)\int_{(\bR^d)^N}\Dist_{MK,1}(\mu_{T^N_tZ_N^{in}}^{\otimes m},(f(t)\scrL^d)^{\otimes m})P_N^{in}(dZ_N^{in})+\|\phi_m\|_{L^\infty}\frac{m(m-1)}{2N}&
\\
\le
m\Lip(\phi_m)\int_{(\bR^d)^N}\Dist_{MK,1}(\mu_{T^N_tZ_N^{in}},(f(t)\scrL^d))P_N^{in}(dZ_N^{in})+\|\phi_m\|_{L^\infty}\frac{m(m-1)}{2N}&
\\
\le\!me^{2L|t|}\Lip(\phi_m)\!\int_{(\bR^d)^N}\!\Dist_{MK,1}(\mu_{Z_N^{in}},(f^{in}\scrL^d))P_N^{in}(dZ_N^{in})\!+\!\|\phi_m\|_{L^\infty}\frac{m(m-1)}{2N}&\,.
\end{aligned}
\end{equation*}
The first inequality above follows from statement (a) in the Theorem and the second inequality from the estimate on $R_{N,m}(t)$. The third inequality above follows from (\ref{DualKR}) and 
the fourth from the corollary above, while the fifth follows from (\ref{Dobr<}). Since $P_{N:m}(t)$ is a probability measure on $(\bR^d)^m$, the bound (\ref{IneqB}) and the chain of inequalities 
above imply statement (b) in the Theorem.

Now for statement (c). First
\begin{equation*}
\begin{aligned}
\int_{(\bR^d)^N}\Dist_{MK,1}&(\mu_{Z_N^{in}},f^{in}\scrL^d)P_N^{in}(dZ_N^{in})
\\
&\le
\int_{(\bR^d)^N}\Dist_{MK,2}(\mu_{Z_N^{in}},f^{in}\scrL^d)P_N^{in}(dZ_N^{in})
\\
&\le
\left(\int_{(\bR^d)^N}\Dist_{MK,2}(\mu_{Z_N^{in}},f^{in}\scrL^d)^2P_N^{in}(dZ_N^{in})\right)^{1/2}
\end{aligned}
\end{equation*}
where the first inequality results from the ordering of Monge-Kantorovich distances (see formula (7.3), section \S 7.1.2 in \cite{VillaniTOT}), while the second follows from the Cauchy-Schwarz 
inequality. By Theorem 1.1 of \cite{HoroKaran},
$$
\int_{(\bR^d)^N}\Dist_{MK,2}(\mu_{Z_N^{in}},f^{in}\scrL^d)^2P_N^{in}(dZ_N^{in})\le\frac{C(a,d)^2}{N^{2/(d+4)}}\,,
$$
since $P_N^{in}=(f^{in}\scrL^d)^{\otimes N}$ with $f^{in}$ satisfying (\ref{d+5Momf}).

Together with (b), this implies the first inequality in (c). The second estimate in (c) is a consequence of the first in the case $m=1$ and of (\ref{DualKR}).
\end{proof}

\smallskip
The convergence rate in Theorem \ref{T-Chaotic} (c) obviously depends on the quantitative chaoticity estimate in \cite{HoroKaran}. More information on analogous chaoticity estimates can 
be found in Lemma 4.2 of \cite{MischMouhot}.

%%%%%%%%%%%%%%%%%%%%%%%%%%%%%%%%%%%%%%%%%%%%%%%%%%%%%%%%%%%%%%%%%%%%%%%%%%%%%%%%%%%%%%%%%%%%%%%

\section{Weak stability of the BBGKY hierarchy}

%%%%%%%%%%%%%%%%%%%%%%%%%%%%%%%%%%%%%%%%%%%%%%%%%%%%%%%%%%%%%%%%%%%%%%%%%%%%%%%%%%%%%%%%%%%%%%%

In this section, we establish a stability property of the BBGKY hierarchy in the weak topology of probability measures. This stability property is uniform as the particle number tends to infinity.

Our stability estimate uses the following variant of Monge-Kantorovich distance. Let $P\in\cP_1((\bR^d)^M)$ and $Q\in\cP_1((\bR^d)^N)$ satisfy the symmetry condition (\ref{SymPin}). 

Equivalently, $P\in\cP_1((\bR^d)^M/\fS_M)$ and $Q\in\cP_1((\bR^d)^N/\fS_N)$, where $(\bR^d)^M/\fS_M$ (resp. $(\bR^d)^N/\fS_N$) is the quotient of $(\bR^d)^M$ under the action of 
$\fS_M$ (resp. of $(\bR^d)^N$ under the action of $\fS_N$) defined by $S_\si$ as in (\ref{DefSsi}). Consider
$$
\DDist_{MK,1}(P,Q)=\inf_{\rho\in\Pi(P,Q)}\iint_{(\bR^d)^M\times(\bR^d)^N}\Dist_{MK,1}(\mu_{X_M},\mu_{Y_N})\rho(dX_M,dY_N)\,.
$$

\begin{theorem}\label{T-StabBBGKY}
Let $M,N\ge 1$, and let $P_M\in\cP_1((\bR^d)^M)$ and $Q_N\in\cP_1((\bR^d)^N)$ satisfy the symmetry condition (\ref{SymPin}). For each $t\in\bR$, set $P_M(t):=T^M_t\#P_M^{in}$ and 
$Q_N(t):=T^N_t\#Q_N^{in}$. Then

\smallskip
\noindent
(a) For each $t\in\bR$, one has
$$
\DDist_{MK,1}(P_M(t),Q_N(t))\le e^{2L|t|}\DDist_{MK,1}(P_M^{in},Q_N^{in})\,.
$$

\noindent
(b) For each $t\in\bR$, each $m,M,N\in\bN^*$ such that $M,N\ge m$, and for each bounded and Lipschitz continuous function $\phi_m$ defined on $(\bR^d)^m$, one has
\begin{equation*}
\begin{aligned}
{}&|\la P_{M:m}(t)-Q_{N:m}(t),\phi_m\ra|
\\
&\le m\left(e^{2L|t|}\Lip(\phi_m)\DDist_{MK,1}(P_M^{in},Q_N^{in})+(m-1)\|\phi_m\|_{L^\infty}\left(\frac1M+\frac1N\right)\right)\,.
\end{aligned}
\end{equation*}
\end{theorem}

\begin{proof}
First we establish statement (a). Let $X_M\in(\bR^d)^M$ and $Y_N\in(\bR^d)^N$; then $t\mapsto\mu_{T^M_tX_M}$ and $t\mapsto\mu_{T^N_tY_N}$ are two weak solutions of the Vlasov 
equation (\ref{VlasovK}). By Dobrushin's estimate 
$$
\Dist_{MK,1}(\mu_{T^M_tX_M},\mu_{T^N_tY_N})\le e^{2L|t|}\Dist_{MK,1}(\mu_{X_M},\mu_{Y_N})\,.
$$
Let $\rho^{in}\in\Pi(P_M^{in},Q_N^{in})$; averaging both sides of this inequality with respect to $\rho^{in}$ gives
\begin{equation*}
\begin{aligned}
\iint_{(\bR^d)^M\times(\bR^d)^N}\Dist_{MK,1}(\mu_{X_M},\mu_{Y_N})\rho(t,dX_M,dY_N)&
\\
=
\iint_{(\bR^d)^M\times(\bR^d)^N}\Dist_{MK,1}(\mu_{T^M_tX_M},\mu_{T^N_tY_N})\rho^{in}(dX_M,dY_N)&
\\
\le
e^{2L|t|}\iint_{(\bR^d)^M\times(\bR^d)^N}\Dist_{MK,1}(\mu_{X_M},\mu_{Y_N})\rho^{in}(dX_M,dY_N)&\,,
\end{aligned}
\end{equation*}
where $\rho(t)$ is the push-forward of $\rho^{in}$ under the map $(X_M,Y_N)\mapsto(T^M_tX_M,T^N_tY_N)$. Therefore $\rho(t)\in\Pi(P_M(t),Q_N(t))$: indeed
\begin{equation*}
\begin{aligned}
\iint_{(\bR^d)^M\times(\bR^d)^N}(\phi(X_M)+\psi(Y_N))\rho(t,dX_M,dY_N)&
\\
=
\iint_{(\bR^d)^M\times(\bR^d)^N}(\phi(T^M_tX_M)+\psi(T^N_tY_N))\rho^{in}(dX_M,dY_N)&
\\
=
\int_{(\bR^d)^M}\phi(T^M_tX_M)P_M^{in}(dX_M)+\int_{(\bR^d)^N}\psi(T^N_tY_N))Q_N^{in}(dY_N)&
\\
=
\int_{(\bR^d)^M}\phi(X_M)P_M(t,dX_M)+\int_{(\bR^d)^N}\psi(Y_N))Q_N(t,dY_N)&\,.
\end{aligned}
\end{equation*}
Therefore
\begin{equation*}
\begin{aligned}
\DDist_{MK,1}(P_M(t),Q_N(t))
\le\iint_{(\bR^d)^M\times(\bR^d)^N}\Dist_{MK,1}(\mu_{X_M},\mu_{Y_N})\rho(t,dX_M,dY_N)&
\\
\le e^{2L|t|}\iint_{(\bR^d)^M\times(\bR^d)^N}\Dist_{MK,1}(\mu_{X_M},\mu_{Y_N})\rho^{in}(dX_M,dY_N)&\,.
\end{aligned}
\end{equation*}
Since this is true for all $\rho^{in}\in\Pi(P_M^{in},Q_N^{in})$, minimizing the right hand side of the inequality above in $\rho^{in}$ establishes the inequality in (a).

As for statement (b), pick $\phi_m$ to be a bounded and Lipschitz continuous function on $(\bR^d)^m$. By formula (\ref{Empir-BBGKY})
\begin{equation*}
\begin{aligned}
\la P_{N:m}(t)-Q_{N:m}(t),\phi_m\ra
&=
\iint_{(\bR^d)^M\times(\bR^d)^N}\la\mu_{T^M_tX_M}^{\otimes m}-\mu_{T^N_tY_N}^{\otimes m},\phi_m\ra\rho^{in}(dX_M,dY_N)
\\
&+\left(1-\frac{M!}{(M-m)!M^m}\right)\la P_{M:m}(t),\phi_m\ra-\la R_{M,m}(t),\phi_m\ra
\\
&-\left(1-\frac{N!}{(N-m)!N^m}\right)\la Q_{N:m}(t),\phi_m\ra+\la S_{N,m}(t),\phi_m\ra
\end{aligned}
\end{equation*}
where $R_{M,m}(t)$ is the Radon measure defined in (\ref{DefRNm}) and the $S_{N,m}(t)$ the analogue of $R_{N,m}(t)$ with $Q_N^{in}$ replacing $P_M^{in}$ in formula (\ref{DefRNm}). 
Thus, by (\ref{IneqRNm}) and (\ref{IneqB}), one has
\begin{equation*}
\begin{aligned}
\left(1-\frac{M!}{(M-m)!M^m}\right)&|\la P_{M:m}(t),\phi_m\ra|+|\la R_{M,m}(t),\phi_m\ra|
\\
&\le 2\left(1-\frac{M!}{(M-m)!M^m}\right)\|\phi_m\|_{L^\infty}\le\frac{m(m-1)}{M}\|\phi_m\|_{L^\infty}\,,
\end{aligned}
\end{equation*}
and, by the same token
\begin{equation*}
\begin{aligned}
\left(1-\frac{N!}{(N-m)!N^m}\right)|\la Q_{N:m}(t),\phi_m\ra|+|\la S_{N,m}(t),\phi_m\ra|&
\\
\le\frac{m(m-1)}{N}\|\phi_m\|_{L^\infty}&\,.
\end{aligned}
\end{equation*}

On the other hand, by Corollary \ref{C-DistProdTens} and Dobrushin's inequality (\ref{Dobr<}), 
\begin{equation*}
\begin{aligned}
|\la\mu_{T^M_tX_M}^{\otimes m}-\mu_{T^N_tY_N}^{\otimes m},\phi_m\ra|\le\Lip(\phi_m)\Dist_{MK,1}(\mu_{T^M_tX_M}^{\otimes m},\mu_{T^N_tY_N}^{\otimes m})&
\\
\le m\Lip(\phi_m)\Dist_{MK,1}(\mu_{T^M_tX_M},\mu_{T^N_tY_N})&
\\
\le m\Lip(\phi_m)e^{2L|t|}\Dist_{MK,1}(\mu_{X_M},\mu_{Y_N})&\,,
\end{aligned}
\end{equation*}
so that
\begin{equation*}
\begin{aligned}
|\la P_{M:m}(t)-Q_{N:m}(t),\phi_m\ra|\le\left(\frac{m(m-1)}{M}+\frac{m(m-1)}{N}\right)\|\phi_m\|_{L^\infty}&
\\
+m\Lip(\phi_m)e^{2L|t|}\iint_{(\bR^d)^M\times(\bR^d)^N}\Dist_{MK,1}(\mu_{X_M},\mu_{Y_N})\rho^{in}(dX_M,dY_N)&\,.
\end{aligned}
\end{equation*}
Minimizing the integral on the right hand side of this inequality as $\rho^{in}$ runs through $\Pi(P_M^{in},Q_N^{in})$ leads to the inequality stated in (b).
\end{proof}

%%%%%%%%%%%%%%%%%%%%%%%%%%%%%%%%%%%%%%%%%%%%%%%%%%%%%%%%%%%%%%%%%%%%%%%%%%%%%%%%%%%%%%%%%%%%%%%%

\section{Continuous dependence on the initial data of statistical solutions of the Vlasov mean field PDE}\lb{S-ContDep}

%%%%%%%%%%%%%%%%%%%%%%%%%%%%%%%%%%%%%%%%%%%%%%%%%%%%%%%%%%%%%%%%%%%%%%%%%%%%%%%%%%%%%%%%%%%%%%%%

In this section, we identify each element $P_N$ of $\cP_1((\bR^d)^N/\fS_N)$ with the element of $\cP(\cP_1(\bR^d))$ defined as the push-forward of $P_N$ under the map
$$
(\bR^d)^N\ni Z_N=(z_1,\ldots,z_N)\mapsto\mu_{Z_N}=\frac1N\sum_{k=1}^N\de_{z_k}\in\cP_1(\bR^d)\,.
$$
Since $\bR^d$ endowed with the Euclidean distance is a complete metric space, $\cP_1(\bR^d)$ endowed with the Monge-Kantorovich distance $\Dist_{MK,1}$ is a Polish space (see Proposition 
7.1.5 in \cite{AGS}). Define 
$$
\cP_1(\cP_1(\bR^d)):=\left\{P\in\cP(\cP_1(\bR^d))\,\hbox{ s.t. }\,\int_{\cP_1(\bR^d)}\Dist_{MK,1}(\mu,\de_0)P(d\mu)<\infty\right\}\,.
$$
Obviously, any element $P_N$ of $\cP_1((\bR^d)^N/\fS_N)$ is identified with an element of $\cP_1(\cP_1(\bR^d))$, since
\begin{equation*}
\begin{aligned}
\int_{\cP_1(\bR^d)}\Dist_{MK,1}(\mu,\de_0)P_N(d\mu)&=\int_{(\bR^d)^N}\Dist_{MK,1}(\mu_{Z_N},\de_0)P_N(dZ_N)
\\
&\le\int_{(\bR^d)^N}\frac1N\sum_{k=1}^N|z_k|P_N(dZ_N)<\infty\,.
\end{aligned}
\end{equation*}
Then the distance $\DDist_{MK,1}$ introduced in the previous section is extended to $\cP_1(\cP_1(\bR^d))$ as follows:
$$
\DDist_{MK,1}(P,Q)=\inf_{\rho\in\Pi(P,Q)}\iint_{\cP_1(\bR^d)\times\cP_1(\bR^d)}\Dist_{MK,1}(\mu,\nu)\rho(d\mu,d\nu)\,.
$$

Let $\cE$ be the set of probability measures $P^{in}\in\cP_1(\cP_1(\bR^d))$ such that
\begin{equation}\label{d+5MomPin}
\int_{\cP_1(\bR^d)}\la f,|z|^{d+5}\ra P^{in}(df)<\infty\,.
\end{equation}

\smallskip
The next result is a consequence of Theorem \ref{T-StabBBGKY}.

\begin{theorem}\label{T-StabInfin}
There exists a unique $1$-parameter group $T^\infty_t$ defined on $\cE$ and satisfying the following properties:

\smallskip
\noindent
(a) For each $P^{in}\in\cE$ and each sequence $P^{in}_N\in\cP_1((\bR^d)^N/\fS_N)$ such that 
$$
\DDist_{MK,1}(P^{in}_N,P^{in})\to 0\quad\hbox{ as }N\to\infty\,,
$$
one has 
$$
\DDist_{MK,1}(T^N_t\#P^{in}_N,T^\infty_tP^{in})\to 0\quad\hbox{ as }N\to\infty\,,
$$
uniformly in $t\in[-T,T]$ for all $T>0$.

\noindent
(b) For each $P^{in}\in\cE$, the map $\bR\ni t\mapsto T^\infty_tP^{in}\in\cP_1(\cP_1(\bR^d))$ is continuous for the distance $\DDist_{MK,1}$.

\noindent
(c) For each $P^{in}\in\cE$, the family $P_m$ defined for all $m\ge 1$ by the formula
$$
P_m(t):=\int_{\cP_1(\bR^d)}T^\infty_tP^{in}(df)f^{\otimes m}
$$
is a solution to the mean field, infinite hierarchy (\ref{VlasovInfin}) in the sense of distributions.

\noindent
(d) For each $P^{in},Q^{in}\in\cE$, one has
$$
\DDist_{MK,1}(T^\infty_tP^{in},T^\infty_tQ^{in})\le e^{2L|t|}\DDist_{MK,1}(P^{in},Q^{in})\,,\qquad t\in\bR\,.
$$
\end{theorem}

\smallskip
The generator of the $1$-parameter group $T^\infty_t$ can be computed explicitly; see Lemma 2.11 of \cite{MischMouhot} for a detailed description of this computation and of the pertaining
functional framework.

\smallskip
The definition of the $1$-parameter group $T^\infty_t$ in statement (a) above requires constructing at least one sequence $P^{in}_N$ converging to $P^{in}$ for the distance $\DDist_{MK,1}$. 
This can be done  as explained in the following lemma.

\begin{lemma}\label{L-QNtoQ}
For each $Q\in\cE$ and each $N\ge 1$, set
$$
Q_N:=\int_{\cP_1(\bR^d)}Q(df)f^{\otimes N}\in\cP_1((\bR^d)^N/\fS_N)\,.
$$
Then the sequence $\DDist_{MK,1}(Q_N,Q)$ converges to $0$ as $N\to\infty$.
\end{lemma}

\smallskip
For each $f\in\cP_{d+5}(\bR^d)$, applying either Theorem 1.1 of \cite{HoroKaran} or the strong law of large numbers shows that the sequence $f^{\otimes N}$ (or equivalently its push-forward 
by the map $(\bR^d)^N\ni Z_N\mapsto\mu_{Z_N}\in\cP_1(\bR^d)$) converges to $\de_f$ for the distance $\DDist_{MK,1}$ as $N\to\infty$. Since any probability measure $Q\in\cE$ can be 
represented as a barycenter of $\de_f$ by the (tautological) formula
$$
Q=\int_{\cP_1(\bR^d)}Q(df)\de_f\,,
$$
it is natural to expect that $Q$ is the limit for $N\to\infty$ of the corresponding barycenters $Q_N$ of $f^{\otimes N}$ defined in Lemma \ref{L-QNtoQ}. The missing details are explained in the 
proof below.

\begin{proof}
Observe that $\rho(dZ_N,df):=f^{\otimes N}(dZ_N)Q(df)$ is a coupling of $Q_N$ and $Q$. Indeed, for each $\phi\in C_b((\bR^d)^N)$ and each $\Phi$ continuous on 
$\cP_1(\bR^d)$ for the distance $\Dist_{MK,1}$, one has
\begin{equation*}
\begin{aligned}
\int_{\cP_1(\bR^d)}\int_{(\bR^d)^N}&(\phi(Z_N)+\Phi(f))f^{\otimes N}(dZ_N)Q(df)
\\
=
&\int_{(\bR^d)^N}\phi(Z_N)\int_{\cP_1(\bR^d)}Q(df)f^{\otimes N}(dZ_N)
\\
&+\int_{\cP_1(\bR^d)}\Phi(f)\int_{(\bR^d)^N}f^{\otimes N}(dZ_N)Q(df)
\\
=
&\int_{(\bR^d)^N}\phi(Z_N)Q_N(dZ_N)+\int_{\cP_1(\bR^d)}\Phi(f)Q(df)\,.
\end{aligned}
\end{equation*}
Therefore
$$
\DDist_{MK,1}(Q_N,Q)\le\int_{\cP_1(\bR^d)}\int_{\bR^d}\Dist_{MK,1}(\mu_{Z_N},f)f^{\otimes N}(dZ_N)Q(df)\,.
$$
By Theorem 1.1 of \cite{HoroKaran}, 
$$
\int_{\bR^d}\Dist_{MK,1}(\mu_{Z_N},f)f^{\otimes N}(dZ_N)\to 0\hbox{ as }N\to\infty
$$
for all $f\in\cP_{d+5}(\bR^d)$, i.e. $Q$-a.e. in $\cP_1(\bR^d)$ in view of (\ref{d+5MomPin}). On the other hand
\begin{equation*}
\begin{aligned}
\Dist_{MK,1}(\mu_{Z_N},f)&\le\frac1N\sum_{k=1}^N\Dist_{MK,1}(\de_{z_k},f)
\\
&\le\frac1N\sum_{k=1}^N\Dist_{MK,1}(\de_{z_k},\de_0)+\Dist_{MK,1}(\de_0,f)
\\
\le\frac1N\sum_{k=1}^N|z_k|+\Dist_{MK,1}(\de_0,f)
\end{aligned}
\end{equation*}
so that
$$
\int_{\bR^d}\Dist_{MK,1}(\mu_{Z_N},f)f^{\otimes N}(dZ_N)\le\la f,|z|\ra+\Dist_{MK,1}(\de_0,f)\,.
$$
(In the first inequality above, we have used the fact that
$$
\Dist_{MK,1}((1-\th)\mu+\th\nu,\l)\le(1-\th)\Dist_{MK,1}(\mu,\l)+\th\Dist_{MK,1}(\nu,\l)
$$
for all $\l,\mu,\nu\in\cP_1(\bR^d)$, which is an obvious consequence of (\ref{DualKR}).) Since $Q\in\cE$, one has
$$
\int_{\cP_1(\bR^d)}(\la f,|z|\ra+\Dist_{MK,1}(\de_0,f))Q(df)<\infty\,.
$$
By dominated convergence
$$
\int_{\cP_1(\bR^d)}\int_{\bR^d}\Dist_{MK,1}(\mu_{Z_N},f)f^{\otimes N}(dZ_N)Q(df)\to 0\,,
$$
as $N\to\infty$, and this concludes the proof.
\end{proof}

\begin{proof}[Proof of Theorem \ref{T-StabInfin}]
Pick a sequence $P^{in}_N$ of elements of $\cP_1((\bR^d)^N/\fS_N)$ for each $N\ge 1$ such that $\DDist_{MK,1}(P^{in}_N,P^{in})\to 0$ as $N\to\infty$. (The construction in 
Lemma \ref{L-QNtoQ} provides one example of sequence $P^{in}_N$ whenever $P^{in}\in\cE$.) For each $t\in\bR$, we define $P_N(t):=T^N_t\#P_N^{in}$; we recall that 
$P_N(t)\in\cP_1((\bR^d)^N/\fS_N)$ by (\ref{SymPNt}). Since $P^{in}_N$ converges to $P^{in}$ for the distance $\DDist_{MK,1}$, it is in particular a Cauchy sequence for that 
distance. Thus $P_N(t)$ is a Cauchy sequence for the distance $\DDist_{MK,1}$ for each $t\in\bR$, by Theorem \ref{T-StabBBGKY} (a). Since $\cP_1(\bR^d)$ endowed 
with the distance $\Dist_{MK,1}$ is a complete space, the set $\cP_1(\cP_1(\bR^d))$ is a complete space for the distance $\DDist_{MK,1}$ (by Proposition 7.1.5 in \cite{AGS}). 
Therefore, there exists a unique map $\bR\ni t\mapsto P(t)\in\cP_1(\cP_1(\bR^d))$ such that $\DDist_{MK,1}(P_N(t),P(t))\to 0$ as $N\to\infty$ for each $t\in\bR$. Besides, the 
estimate in Theorem \ref{T-StabBBGKY} (a) shows that this convergence is uniform in $t\in[-T,T]$ for each $T>0$. Since $P_N$ is a continuous map on $\bR$ with values in 
$\cP_1(\cP_1(\bR^d))$ for the distance $\DDist_{MK,1}$, we conclude that $P$ is also continuous on $\bR$ with values in $\cP_1(\cP_1(\bR^d))$ for the distance $\DDist_{MK,1}$. 
This proves (a) and (b).

Let $P^{in},Q^{in}\in\cP_1(\cP_1(\bR^d))$, and let $P_N^{in}$ and $Q_N^{in}$ be sequences of elements of $\cP_1((\bR^d)^N/\fS_N)$ converging respectively to $P^{in}$ 
and $Q^{in}$ in $\cP_1(\cP_1(\bR^d))$ for the distance $\DDist_{MK,1}$ as $N\to\infty$. As explained in the discussion above, the sequences $P_N(t):=T^N_t\#P_N^{in}$ and 
$Q_N(t):=T^N_t\#Q_N^{in}$ converge to $P(t):=T^\infty_tP^{in}$ and $Q(t)=T^\infty_tQ^{in}\in\cP_1(\cP_1(\bR^d))$ uniformly in $t\in[-T,T]$ for all $T>0$ as  $N\to\infty$, and 
passing to the limit as $N\to\infty$ in the inequality of Theorem \ref{T-StabBBGKY} implies that
$$
\DDist_{MK,1}(P(t),Q(t))\le e^{2L|t|}\DDist_{MK,1}(P^{in},Q^{in})\,,\qquad t\in\bR\,.
$$
This proves (d). In particular $P(t)=Q(t)$ for all $t\in\bR$ if $P^{in}=Q^{in}$. In other words, the function $t\mapsto P(t)$ is uniquely determined by the initial condition $P^{in}$. 

Since $\DDist_{MK,1}(P_N(t),P(t))\to 0$ as $N\to\infty$, one has
$$
\int_{\cP_1(\bR^d)}\la f,(|z|\wedge R)^{d+5}\ra P_N(t,df)\to\int_{\cP_1(\bR^d)}\la f,(|z|\wedge R)^{d+5}\ra P(t,df)
$$
as $N\to\infty$ for all $R\ge 1$. Indeed, the map $f\mapsto\la f,(|z|\wedge R)^{d+5}\ra$ is Lipschitz continuous on $\cP_1(\bR^d)$ for the distance $\Dist_{MK,1}$, with
$$
|\la f-g,(|z|\wedge R)^{d+5}\ra|\le(d+5)R^{d+4}\Dist_{MK,1}(f,g)\,.
$$
On the other hand
\begin{equation*}
\begin{aligned}
\int_{\cP_1(\bR^d)}\la f,(|z|\wedge R)^{d+5}\ra P_N(t,df)&\le\int_{\cP_1(\bR^d)}\la f,|z|^{d+5}\ra P_N(t,df)
\\
&=\La\int_{\cP_1(\bR^d)}P_N(t,df)f,|z|^{d+5}\Ra
\\
&=\la P_{N:1}(t),|z|^{d+5}\ra
\\
&\le e^{2L(d+5)|t|}\la P_{N:1}(0),|z|^{d+5}\ra
\\
&=e^{2L(d+5)|t|}\int_{\cP_1(\bR^d)}\la f,|z|^{d+5}\ra P^{in}(df)
\end{aligned}
\end{equation*}
for all $N,R\ge 1$ and all $t\in\bR$ by (\ref{rMomPNm}). Applying Fatou's lemma shows that $P(t)\in\cE$ for all $t\in\bR$ whenever $P^{in}\in\cE$. Since 
$$
T^N_{t+s}\#P^{in}_N=T^N_t\#(T^N_s\#P^{in}_N)\quad\hbox{ for all }t,s\in\bR\hbox{ and all }N\ge 1\,,
$$
we conclude that $T^\infty_t$ defines a $1$-parameter group on $\cE$.
\end{proof}

\smallskip

\smallskip
To conclude this section, let us compare Theorem \ref{T-StabInfin} and Spohn's theorem \cite{Spohn}. 

Theorem \ref{T-StabInfin} implies that weak solutions of the infinite mean field hierarchy obtained as limits of the BBGKY hierarchy are uniquely determined by their initial condition. 

Spohn's theorem \cite{Spohn} states that any weak solution to the infinite mean field hierarchy that is weakly differentiable in the time variable $t$ is uniquely determined by its initial 
condition. To be more precise, a time-dependent sequence 
$$
\bR\ni t\mapsto(P_m(t))_{m\ge 1}\in\prod_{m\ge 1}\cP((\bR^d)^m/\fS_m)^{\bR}
$$ 
is said to be weakly differentiable if the map
$$
t\mapsto\int_{(\bR^d)^m}\phi_m(z_1,\ldots,z_m)P_m(t,dz_1\ldots dz_m)
$$
is differentiable on $\bR$ for all $m\ge 1$ and all $\phi_m\in C^1_c((\bR^d)^m)$. Such a sequence is said to be a weak solution to the infinite mean field hierarchy if
\begin{equation*}
\begin{aligned}
\frac{d}{dt}\int_{(\bR^d)^m}\phi_m(z_1,\ldots,z_m)P_m(t,dz_1\ldots dz_m)&
\\
=
\int_{(\bR^d)^m}\sum_{k=1}^mK(z_k,z_{m+1})\cdot\grad_{z_k}\phi_m(z_1,\ldots,z_m)P_{m+1}(t,dz_1\ldots dz_{m+1})&\,,
\end{aligned}
\end{equation*}
for all $t\in\bR$, all $m\ge 1$ and all $\phi_m\in C^1_c((\bR^d)^m)$.

In other words, let $(P_m)_{m\ge 1}$ be a weak solution to the infinite mean field hierarchy that is weakly differentiable in $t$ and satisfies $P_m(0)\in\cP_1((\bR^d)^m/\fS_m)$ for each 
$m\ge 1$. By the Hewitt-Savage theorem \cite{HewSav}, let $P^{in}$ be the unique element of $\cP(\cP_1(\bR^d))$ such that
\begin{equation}\label{CondInInfinHier}
P_m(0)=\int_{\cP_1(\bR^d)}P^{in}(df)f^{\otimes m}\quad\hbox{ for all }m\ge 1\,.
\end{equation}
Finally, let $V_t$ be the $1$-parameter group defined on $\cP_1(\bR^d)$ by the relation 
$$
V_t\mu^{in}:=\mu(t)\,,
$$
where $\mu$ is the solution to the Cauchy problem (\ref{VlasovMut}). Then one has
\begin{equation}\label{DefPmt}
P_m(t)=\int_{\cP_1(\bR^d)}P^{in}(df)(V_tf)^{\otimes m}
\end{equation}
for all $t\in\bR$ and all $m\ge 1$. (Indeed, the sequence $((V_tf)^{\otimes m})_{m\ge 1}$ is a solution to the infinite mean field hierarchy for each $f\in\cP_1(\bR^d)$. Hence, by linearity, the 
sequence on the right hand side of the formula above defines a weak solution to the infinite mean field hierarchy that is weakly differentiable in time and coincides with $(P_m(t))_{m\ge 1}$
for $t=0$. By Spohn's uniqueness theorem, these sequences must coincide for all $t\in\bR$.)

In the language of Theorem \ref{T-StabInfin}, Spohn's uniqueness theorem implies that 
\begin{equation}\label{EqSpohn}
T^\infty_tP^{in}=V_t\#P^{in}\,,\quad\hbox{ for all }t\in\bR\hbox{ and all }P^{in}\in\cE\,.
\end{equation}
Notice that Spohn's uniqueness theorem is not a consequence of our proof of Theorem \ref{T-StabInfin} (d). Since the stability inequalities in that theorem are obtained by passing to the large 
$N$ limit in the corresponding inequalities for the BBGKY hierarchy in Theorem \ref{T-StabBBGKY}, they apply only to solutions of the infinite hierarchy that are obtained as limits for $N\to\infty$ 
of solutions of the $N$-particle BBGKY hierarchy.

Of course, with the additional information contained in Spohn's uniqueness theorem, the stability estimate in Theorem \ref{T-StabInfin} (d) holds for all weak solutions of the infinite mean field 
hierarchy that are weakly differentiable in time, and whose initial condition is defined by (\ref{CondInInfinHier}) with $P^{in}\in\cE$.

Let us explain how the discussion in this section is related to the notion of ``statistical solution'' of the mean field PDE (\ref{VlasovK}). This notion is very clearly explained by Spohn (see
\cite{Spohn} on p. 448, especially formulas (1.16)-(1.19)). We briefly recall Spohn's point of view for the reader's convenience. Let $v\in C^1(\bR^n;\bR^n)$ satisfy $|v(x)|=O(|x|)$ as 
$|x|\to\infty$; by the Cauchy-Lipschitz theorem, the vector field $v$ generates a global flow $X_t$ on $\bR^N$. In other words, for each $x\in\bR^N$, the solution of the Cauchy problem
\begin{equation}\label{ODE}
\dot{X}=v(X)\,,\qquad X(0)=x
\end{equation}
is the trajectory $t\mapsto X_t(x)$ going through $x$ at time $t=0$. Consider instead of a single initial point $x$ a cloud of initial data distributed under $p^{in}\in\cP(\bR^n)$. It is transported 
by the flow into a cloud of points which, at time $t$, are distributed under $X_t\#p^{in}$. It is therefore natural to think of the map $t\mapsto X_t\#p^{in}$ as a ``statistical solution'' of the ODE
(\ref{ODE}). If one replaces $\bR^n$ with $\cP_1(\bR^d)$ and the ordinary differential equation (\ref{ODE}) with the Vlasov equation (\ref{VlasovMut}), it is equally natural to think of the map
$t\mapsto V_t\#P^{in}$ as the statistical solution of the Vlasov equation (\ref{VlasovMut}) starting from $P^{in}\in\cE$ at time $t=0$. Then the equality (\ref{EqSpohn}) implies that our Theorem
\ref{T-StabInfin} establishes the Lipschitz continuous dependence on the initial data of statistical solutions of the Vlasov equation in the distance $\DDist_{MK,1}$.

%%%%%%%%%%%%%%%%%%%%%%%%%%%%%%%%%%%%%%%%%%%%%%%%%%%%%%%%%%%%%%%%%%%%%%%%%%%%%%%%%%%%%%%%%%%%%%%%

\section{Appendix: Spohn's uniqueness theorem}

%%%%%%%%%%%%%%%%%%%%%%%%%%%%%%%%%%%%%%%%%%%%%%%%%%%%%%%%%%%%%%%%%%%%%%%%%%%%%%%%%%%%%%%%%%%%%%%%

For the sake of being complete, we recall Spohn's uniqueness theorem \cite{Spohn}, and briefly sketch its proof.

\begin{theorem}
Let $\bR\ni t\mapsto P(t)\in\cE$ be such that the sequence $t\mapsto(P_m)_{m\ge 1}$ defined by 
\begin{equation}\label{Def2Pmt}
P_m(t):=\int_{\cP_1(\bR^d)}f^{\otimes m}P(t,df)
\end{equation}
is a weakly differentiable in $t\in\bR$ weak solution to the infinite mean field hierarchy (\ref{VlasovInfin}). Then
$$
P(t)=V_t\#P(0)\quad\hbox{ for all }t\in\bR
$$
where $V_t$ is the $1$-parameter group such that $t\mapsto V_t\mu^{in}$ is the solution to the Cauchy problem (\ref{VlasovMut}).
\end{theorem}

\smallskip
Returning to the analogy between the ODE (\ref{ODE}) and the mean field PDE (\ref{VlasovK}) recalled at the end of section \ref{S-ContDep}, weakly differentiable weak solutions of the infinite 
mean field hierarchy are analogous to weak solutions of the transport equation
\begin{equation}\label{TranspPDE}
\d_tp+\Div_x(pv)=0\,,\qquad p\rstr_{t=0}=p^{in}\,.
\end{equation}
In this analogy, Spohn's uniqueness theorem is analogous to the method of characteristics, from which we conclude that the unique weak solution of (\ref{TranspPDE}) is given by the formula 
$p(t)=X_t\#p^{in}$. Thus Spohn's theorem can be rephrased as follows: any weakly differentiable weak solution $(P_m)_{m\ge 1}$ of the infinite mean field hierarchy is represented by the 
Hewitt-Savage theorem (i.e. formula (\ref{Def2Pmt})) in terms of a unique $\cE$-valued map $t\mapsto P(t)$ that is is a statistical solution of the mean field PDE.

\smallskip
The argument below is essentially similar to the original proof on pp. 449--453 in \cite{Spohn}, although slightly simpler in places.

\begin{proof}
For each $m\ge 1$ and each $\phi_m\in C_b((\bR^d)^m)$, we denote by $M_m[\phi_m]$ the monomial defined on $\cP_1(\bR^d)$ by the formula
$$
M_m[\phi_m](f):=\int_{(\bR^d)^m}\phi_m(z_1,\ldots,z_m)f(dz_1)\ldots f(dz_m)\,.
$$
For $m=0$, we use the notation $M_{0}(f):=1$ for all $f\in\cP_1(\bR^d)$. 

We also need the notation
$$
A_n[\zeta]:=\sum_{k=1}^nK(z_k,\zeta)\cdot\grad_{z_k}\,,\quad n\ge 1\,,
$$
together with
\begin{equation}\label{DefLn}
L_n:=\frac1n\sum_{1\le k,l\le n}K(z_k,z_l)\cdot\grad_{z_k}=\frac1n\sum_{l=1}^nA_n[z_l]\,.
\end{equation}
Thus, the $m$-th equation in the infinite mean field hierarchy (\ref{VlasovInfin}) takes the form
$$
\d_tP_m(t)=\int_{\bR^d}A_n^*[z_{m+1}]P_{m+1}(dz_{m+1})\,,
$$
where $A_n^*[\zeta]$ designates the (formal) adjoint of $A_n[\zeta]$.

For each $n\ge m\ge 1$ and each $\phi_m\in C^1_c((\bR^d)^m)$, consider the function
$$
\Phi_n(s,t):=\la P(t),M_n[\phi_m\circ T^n_s]\ra=\int_{(\bR^d)^n}\phi_m\circ T^n_s(Z_n)P_n(t,dZ_n)\,.
$$
(In the expression $\phi_m\circ T^n_s$, the function $\phi_m$ is viewed as function of $n\ge m$ variables.)

Observe first that
\begin{equation}\label{dsPhin}
\d_s\Phi_n(s,t)=\la P(t),M_n[\d_s(\phi_m\circ T^n_s)]\ra=\la P(t),M_n[L_n(\phi_m\circ T^n_s)]\ra\,.
\end{equation}
On the other hand, since $P_m$ is a weak solution to the $m$-th equation in the infinite mean field hierarchy (\ref{VlasovInfin}), one has
\begin{equation*}
\begin{aligned}
\d_t\Phi_n(s,t)&=\int_{(\bR^d)^{n+1}}A_n[z_{n+1}](\phi_m\circ T^n_s)(Z_n)P_{n+1}(t,dZ_{n+1})
\\
&=\la P(t),M_{n+1}[A_n(\phi_m\circ T^n_s)]\ra\,.
\end{aligned}
\end{equation*}
Splitting the monomial $M_{n+1}[A_n(\phi_m\circ T^n_s)](f)$ into
\begin{equation}\label{DecompMn+1}
\begin{aligned}
M_{n+1}[A_n(\phi_m\circ T^n_s)](f)&
\\
=\int_{(\bR^d)^n}\int_{\bR^d}A_n[z_{n+1}](\phi_m\circ T^n_s)(Z_n)\mu_{Z_n}(dz_{n+1})f^{\otimes n}(dZ_n)&
\\
+
\int_{(\bR^d)^n}\int_{\bR^d}A_n[z_{n+1}](\phi_m\circ T^n_s)(Z_n)(f-\mu_{Z_n})(dz_{n+1})f^{\otimes n}(dZ_n)&\,,
\end{aligned}
\end{equation}
and observing that
\begin{equation}\label{1sttermMn+1}
\begin{aligned}
\int_{(\bR^d)^n}\int_{\bR^d}A_n[z_{n+1}](\phi_m\circ T^n_s)(Z_n)\mu_{Z_n}(dz_{n+1})f^{\otimes n}(dZ_n)&
\\
=
\int_{(\bR^d)^n}L_n(\phi_m\circ T^n_s)(Z_n)f^{\otimes n}(dZ_n)=M_n[L_n(\phi_m\circ T^n_s)]&\,,
\end{aligned}
\end{equation}
we conclude from (\ref{dsPhin}) that
\begin{equation}\label{dstPhin}
\begin{aligned}
(\d_t-\d_s)\Phi_n(s,t)&
\\
=\int_{\cP_1(\bR^d)}\int_{(\bR^d)^n}\int_{\bR^d}A_n[z_{n+1}](\phi_m\circ T^n_s)(Z_n)(f-\mu_{Z_n})(dz_{n+1})f^{\otimes n}(dZ_n)P(t,df)&\,.
\end{aligned}
\end{equation}

First
\begin{equation*}
\begin{aligned}
M_n[\phi_m\circ T^n_s](f)&=\int_{(\bR^d)^n}\phi_m(Z_m)(T^n_s\#f^{\otimes n})(dZ_n)
\\
&=\int_{(\bR^d)^m}\phi_m(Z_m)(T^n_s\#f^{\otimes n})_{:m}(dZ_m)
\\
&\to\int_{(\bR^d)^m}\phi_m(Z_m)(V_sf)^{\otimes m}(dZ_m)
\end{aligned}
\end{equation*}
as $n\to\infty$, by Theorem \ref{T-Chaotic} (b). Hence
\begin{equation}\label{PhintoPhi}
\begin{aligned}
\Phi_n(s,t)\to&\int_{\cP_1(\bR^d)}\int_{(\bR^d)^m}\phi_m(Z_m)(V_sf)^{\otimes m}(dZ_m)P(t,df)
\\
&=\la V_s\#P(t),M_m[\phi_m]\ra=:\Phi(s,t)
\end{aligned}
\end{equation}
for all $s,t\in\bR$ as $n\to\infty$. Besides $\|\Phi_n\|_{L^\infty(\bR^2)}\le\|\phi_m\|_{L^\infty((\bR^d)^m)}$ for all $n\ge 1$ so that the convergence(\ref{PhintoPhi}) holds in the sense of 
distributions on $\bR^2$.

On the other hand
\begin{equation*}
\begin{aligned}
\left|\int_{(\bR^d)^n}\int_{\bR^d}A_n[z_{n+1}](\phi_m\circ T^n_s)(Z_n)(f-\mu_{Z_n})(dz_{n+1})f^{\otimes n}(dZ_n)\right|&
\\
=
\left|\int_{(\bR^d)^n}\int_{\bR^d}\sum_{k=1}^n\d_{z_k}(\phi_m\circ T^n_s)(Z_n)\cdot\cK(f-\mu_{Z_n})(z_k)f^{\otimes n}(dZ_n)\right|&
\\
\le
\left\|\sum_{k=1}^n|\d_{z_k}(\phi_m\circ T^n_s)\right\|_{L^\infty((\bR^d)^n)}\int_{(\bR^d)^n}\int_{\bR^d}|\sup_{z\in\bR^d}|\cK(f-\mu_{Z_n})(z_k)|f^{\otimes n}(dZ_n)&
\\
\le
L\left\|\sum_{k=1}^n|\d_{z_k}(\phi_m\circ T^n_s)\right\|_{L^\infty((\bR^d)^n)}\int_{(\bR^d)^n}\int_{\bR^d}\Dist_{MK,1}(f,\mu_{Z_n})f^{\otimes n}(dZ_n)&\,.
\end{aligned}
\end{equation*}
Since $P(t)\in\cE$, by the same argument as in the proof of Lemma \ref{L-QNtoQ}, one has
$$
\int_{\cP_1(\bR^d)}\int_{(\bR^d)^n}\int_{\bR^d}\Dist_{MK,1}(f,\mu_{Z_n})f^{\otimes n}(dZ_n)P(t,df)\to 0
$$
as $n\to\infty$. Hence
$$
(\d_s-\d_t)\Phi_n(s,t)\to 0\quad\hbox{�uniformly in }(s,t)\in[-T,T]\times\bR
$$
for all $T>0$ as $n\to\infty$, provided that
\begin{equation}\label{BoundSpohn}
\sup_{|s|\le T}\left\|\sum_{k=1}^n|\d_{z_k}(\phi_m\circ T^n_s)\right\|_{L^\infty((\bR^d)^n)}=O(1)\hbox{ as }n\to\infty\,.
\end{equation}
 
Taking this for granted, we have proved\footnote{By considering $\Phi_n$ instead of $\Phi$, we have avoided computing explicitly $\d_t\Phi$ as in \cite{Spohn}. The identity 
(\ref{1sttermMn+1}) shows that the first term in the decomposition (\ref{DecompMn+1}) exactly cancels with the expression under the bracket on the right hand side of (\ref{dsPhin}). Thus 
we do not need to pass to the large $n$ limit in these expression, so that the discussion in formulas (2.11-17) and (2.27) of \cite{Spohn} becomes useless. The analysis in formulas (2.17-24)
in \cite{Spohn} is roughly equivalent to our proof that $(\d_t-\d_s)\Phi_n\to 0$. Our proof of (\ref{BoundSpohn}) is essentially equivalent to the analysis in formulas (2.18-22) of \cite{Spohn}.} 
that
$$
(\d_s-\d_t)\Phi(s,t)=0\,,\quad (s,t)\in\bR^2\,,
$$
on account of (\ref{PhintoPhi}). Thus the function $s\mapsto\Phi(s,t-s)$ is constant for all $t\in\bR$ and
\begin{equation}\label{SpohnEqual}
\la P(t),M_m[\phi_m]\ra=\Phi(t,0)=\Phi(0,t)=\la V_t\#P(0),M_m[\phi_m]\ra\,.
\end{equation}
By the Stone-Weierstrass theorem, the algebra
$$
\bR\oplus\Span\{M_m[\phi_m]\hbox{ s.t. }m\ge 1\hbox{ and }\phi_m\in C^1_c((\bR^d)^m)\}
$$
is dense in $C(\cP_1(\bR^d))$, so that (\ref{SpohnEqual}) implies the conclusion of Spohn's theorem.

It remains to establish the estimate (\ref{BoundSpohn}) above. By the chain rule
$$
\d_{z_k}(\phi_m\circ T^n_s)(Z_n)=\sum_{l=1}^m\d_{z_l}\phi_m(T^n_sZ_n)a_{lk}(s,Z_n)
$$
where $a_{lk}(s,Z_n)$ is the entry of the Jacobian matrix of $T^n_s$ at $Z_n$ on the $l$-th row and the $k$-th column. Differentiating each side of the $l$-th equation in (\ref{NODE}) with 
respect to the $k$-th component of the initial condition, one arrives at the inequality satisfied by the quantity $\a_{lk}(s):=\|a_{lk}(s,\cdot)\|_{L^\infty((\bR^d)^n)}$:
$$
\a_{lk}(t)\le\de_{lk}+\frac{L}n\sum_{j=1}^n\int_0^t(\a_{lk}(s)+\a_{jk}(s))ds\,,\quad k,l=1,\ldots,n\,,\,\,t\in\bR\,.
$$
In particular
$$
\b_k(t):=\frac1n\sum_{l=1}^n\a_{lk}(t)
$$
satisfies the inequality
$$
\b_k(t)\le\frac1n+2L\int_0^t\b_k(s)ds\,,\quad k=1,\ldots,n\,,\,\,t\in\bR\,,
$$
so that
$$
\b_k(t)\le\frac1ne^{2L|t|}
$$
by Gronwall's lemma. Therefore
\begin{equation*}
\begin{aligned}
\a_{lk}(t)&\le\de_{lk}+L\int_0^t\b_k(s)ds+L\int_0^t\a_{lk}(s)ds
\\
&\le\de_{lk}+\frac{e^{2L|t|}-1}{2n}+L\int_0^t\a_{lk}(s)ds\,,
\end{aligned}
\end{equation*}
and applying Gronwall's lemma again leads to the bound
$$
\a_{lk}(t)\le\de_{lk}e^{L|t|}+\frac1{2n}{e^{3L|t|}}\quad k=1,\ldots,n\,,\,\,t\in\bR\,.
$$
Thus
\begin{equation*}
\begin{aligned}
\|\d_{z_k}(\phi_m\circ T^n_s)(Z_n)\|_{L^\infty((\bR^d)^n)}\le\sum_{l=1}^m\a_{lk}(s)\|\d_{z_l}\phi_m\|_{L^\infty((\bR^d)^m)}&
\\
\le e^{L|s|}\|\d_{z_k}\phi_m\|_{L^\infty((\bR^d)^m)}+\frac1{2n}{e^{3L|t|}}\sum_{l=1}^m\|\d_{z_l}\phi_m\|_{L^\infty((\bR^d)^m)}&\,,
\end{aligned}
\end{equation*}
so that
$$
\left\|\sum_{k=1}^n|\d_{z_k}(\phi_m\circ T^n_s)\right\|_{L^\infty((\bR^d)^n)}
\le
(e^{L|s|}+\tfrac12e^{3L|s|})\sum_{k=1}^m\|\d_{z_k}\phi_m\|_{L^\infty((\bR^d)^m)}\,.
$$
\end{proof}

%%%%%%%%%%%%%%%%%%%%%%%%%%%%%%%%%%%%%%%%%%%%%%%%%%%%%%%%%%%%%%%%%%%%%%%%%%%%%%%%%%%%%%%%%%%%%%%%

%%%%%%%%%%%%%%%%%%%%%%%%%%%%%%%%%%%%%%%%%%%%%%%%%%%%%%%%%%%%%%%%%%%%%%%%%%%%%%%%%%%%%%%%%%%%%%%%
\end{document}